\documentclass[reqno, 11pt]{amsart}
  \makeatletter
  \let\origsection=\section 
  \def\section{\@ifstar{\origsection*}{\mysection}}
  \def\mysection{\@startsection{section}{1}\z@{.7\linespacing\@plus\linespacing}{.5\linespacing}{\normalfont\scshape\centering\S}}
  \makeatother

  \usepackage{amsmath,amssymb,amsthm}
  \usepackage{mathrsfs}
  \usepackage{dsfont}
  \usepackage{mathabx}\changenotsign
  \usepackage[utf8]{inputenc}
  \usepackage{microtype}
  \usepackage{verbatim}
  \usepackage{todonotes}
  \usepackage{enumerate}
  \usepackage{cancel}
  \usepackage{enumitem}

  \usepackage{xcolor}
  \usepackage[backref=page]{hyperref}
  \hypersetup{%
      colorlinks,
      linkcolor={red!60!black},
      citecolor={green!60!black},
      urlcolor={blue!60!black}
  }

\let\phi=\varphi
\def\red{\text{\rm red}}
\def\blue{\text{\rm blue}}
\def\green{\text{\rm green}}

  \usepackage{bookmark}
  \usepackage{booktabs}

  \usepackage[abbrev,msc-links,backrefs]{amsrefs}
  \usepackage{doi}

  \renewcommand{\PrintDOI}[1]{\doi{#1}}

  \usepackage[T1]{fontenc}
  \usepackage{lmodern}

  \usepackage[english]{babel}

  \usepackage{tikz}
  \usetikzlibrary{matrix}

  \usepackage{geometry}
  \let\polishlcross=\l
  \def\l{\ifmmode\ell\else\polishlcross\fi}

  \renewcommand{\setminus}{\smallsetminus}

  \makeatletter
  \def\moverlay{\mathpalette\mov@rlay}
  \def\mov@rlay#1#2{\leavevmode\vtop{%
     \baselineskip\z@skip \lineskiplimit-\maxdimen
     \ialign{\hfil$\m@th#1##$\hfil\cr#2\crcr}}}
  \newcommand{\charfusion}[3][\mathord]{
      #1{\ifx#1\mathop\vphantom{#2}\fi
          \mathpalette\mov@rlay{#2\cr#3}
        }
      \ifx#1\mathop\expandafter\displaylimits\fi}
  \makeatother

  \newtheorem{theorem}{Theorem}
  \newtheorem{lemma}[theorem]{Lemma}
  \newtheorem{corollary}[theorem]{Corollary}
  
  \newtheorem*{remark*}{Remark}

  \newtheorem{fact}[theorem]{Fact}

\begin{document}

\title[Counting Gallai $k$-colorings of complete graphs]{The number of Gallai $k$-colorings of complete graphs}

\author[J.~de~O.~Bastos]{Josefran de Oliveira Bastos}
\address{Engenharia da Computação, Universidade Federal do Ceará, Ceará, Brazil}
\email{josefran@ufc.br}

\author[F.~S.~Benevides]{Fabrício Siqueira Benevides}
\address{Departamento de Matemática, Universidade Federal do Ceará, Ceará, Brazil}
\email{fabricio@mat.ufc.br}

\author[J.~Han]{Jie Han}
\address{Department of Mathematics, University of Rhode Island, Kingston, RI, USA}
\email{jie\_han@uri.edu}

\thanks{The first and second author were supported by CAPES Probral (Proc. 88887.143992/2017-00) and CNPq (Proc. 437841/2018-9); the second author by CNPQ (Proc. 310512/2015-8, Proc. 401519/2016-3) and FUNCAP. The third author was supported by FAPESP (Proc. 2014/18641-5). This study was financed in part by the Coordenação de Aperfeiçoamento de Pessoal de Nível Superior - Brasil (CAPES) - Finance Code 001.}

\begin{abstract}
An edge coloring of the $n$-vertex complete graph, $K_n$, is a \emph{Gallai coloring} if it does not contain any rainbow triangle, that is, a triangle whose edges are colored with three distinct colors. 
We prove that for $n$ large and every $k$ with $k\le 2^{n/4300}$, the number of Gallai colorings of $K_n$ that use at most $k$ given colors is $(\binom{k}{2}+o_n(1))\,2^{\binom{n}{2}}$. 
Our result is asymptotically best possible and implies that, for those $k$, almost all Gallai $k$-colorings use only two colors. However, this is not true for $k \ge \Omega (2^{2n})$.
\end{abstract}

\keywords{Gallai colorings, rainbow triangles, complete graphs, counting}

\maketitle

\section{Introduction}

An edge coloring of the complete graph on $n$ vertices, $K_n$, is a \emph{Gallai coloring} if it contains no \mbox{\emph{rainbow}~$K_3$},  that is, no copy of $K_3$ in which all edges have different colors. Here, we always use $k$-coloring to refer to an edge coloring that uses (not necessarily all) colors from a set of $k$ colors.

The term \emph{Gallai coloring} was used by Gyárfás and Simonyi in~\cites{GySi2004edge}, but those colorings have also been studied under the name \emph{Gallai partitions} by Körner, Simonyi and Tuza in~\cites{korner1992perfect}. The nomenclature is due to a close relation to a result in the influential Gallai's original paper \cites{gallai1967transitiv} -- translated to English with added comments in \cites{gallai1967-translation}. The above mentioned papers are mostly concerned with structural and Ramsey-type results about Gallai colorings.

Following a recent trend of working on problems about counting certain colorings and analyzing the typical structure of them, for integers $k\geq 1$ and $n \geq 2$, we are interested in the problem of counting the number, $c(n, k)$, of Gallai $k$-colorings of $K_n$ (that use a given set of $k$ colors). In this counting, we consider the vertices of $K_n$ as labeled. 

The related problem of counting colorings of graphs such that every color class does not contain a particular graph $F$ was studied originally by Erd\H{o}s and Rothschild (see, e.g.,~\cites{erdos1974some}) has motivated a number of results (for example~\cites{alon2004number,hoppen2015rainbow, hoppen2017graphs, pikhurko2017erdHos}) and it was generalized to colorings that avoid other coloring pattern in \cites{benevides2017edge}. In turn, Gallai colorings have surprising relations to Information Theory \cites{korner2000graph} and a generalization  of the (weak) Perfect Graph Theorem~\cites{cameron1986note}, and has also been generalized to non-complete graphs \cites{gyarfas2010gallai} and hypergraphs \cites{chua2013gallai}. Recently, there has also been a trend in Ramsey-type problems that involve Gallai coloring (see \cites{wu2018all, fox2015erdHos, fujita2011gallai, chen2018gallai, zhang2018gallai}).

In this context, the problem of counting Gallai colorings is very a natural one and has been ``almost'' overlooked. A trivial lower bound for $c(n,k)$, when $k\ge 2$, is given by the colorings that use at most two colors: $c(n,k) \ge \binom{k}{2} (2^{\binom{n}{2}}-2) + k$.

In \cites{bastos2018counting}, its was proved that $c(n,3) \le 7(n+1)\,2^{\binom{n}{2}}$, for every $n\ge 2$. Previous bounds for $c(n,3)$ were obtained by Falgas-Ravry, O’Connell, Strömberg and Uzzell, of order $2^{(1+o_n(1))\binom{n}{2}}$ using the container and entropy method, and around the same time Benevides, Hoppen and Sampaio \cites{benevides2017edge} gave a bound of order $(n-1)!\,2^{\binom{n}{2}}$.

The main purpose of this paper is to prove that the lower bound for $c(n,k)$ is asymptotically sharp for $k \le 2^{n/4300}$.
We point out that, in \cite{GySi2004edge}, it was proved that any Gallai coloring of $K_n$ uses at most $n-1$ colors. In spite of this, the definition of $c(n,k)$ makes sense even for $k \ge n$, as in a Gallai $k$-coloring we do not have to use all $k$ colors. 

\begin{theorem}~\label{thm:main_result}
  For $n$ large enough and every $k$ with $2<k \le 2^{n/4300}$, we have
  \[
   c(n, k) = \left(\binom{k}{2}+o_n(1)\right)\,2^{\binom{n}{2}}. 
  \]
\end{theorem}

This shows that almost all Gallai $k$-colorings use only two colors, for $k \le 2^{n/4300}$. On the other hand, it is easy to see that, for $k\ge 2^{2n}$, this is not true. In fact, fix an edge of $K_n$, say $uv$. Choose a sequence of $3$ distinct colors, say green, red and blue. Color $uv$ by green, all other edges incident to $u$ or $v$ by red, and color the edges induced by $V(K_n) -\{u, v\}$ with either red or blue so that at least one of them is blue. This gives, for $k, n \ge 5$, at least
\[
  k(k-1)(k-2)\left(2^{\binom{n-2}{2}}-1\right) \ge \frac{k(k-1)(k-2)}{2}2^{\binom{n-2}{2}} = \frac{8(k-2)}{2^{2n}}\binom{k}{2}2^{\binom n2}
\]
colorings that use three different colors. Thus, Theorem~\ref{thm:main_result} does not hold for $k =\Omega(2^{2n})$.
Moreover, for $k\gg 2^{2n}$, \emph{most} colorings do \emph{not} use only two colors.

Our proof is self contained with the exception of two (elementary) results from \cites{GySi2004edge} (whose proofs are also short and we encourage the reader to look at them). Furthermore, our proofs are elementary, based on how to classify the colorings while counting the number of ways to extend them. It came recently to our attention that, independently, Balogh and Li~\cites{balogh2018typical} proved a similar upper bound, but for $k$ constant and $n$ large using the container and stability method.


In this paper, we also show (see Corollary~\ref{cor:cnk}) that 
\[
  c(n, k) \le (k-1)^n\,2^{\binom{n}{2}},
\] for every $n\ge 2$ and $k\ge 2$.

%


The following two results from~\cites{GySi2004edge} shall be useful to us. 

\begin{theorem}[see Theorem 2.1 or 2.2 of \cites{GySi2004edge}]\label{thm:GyarfasTree}
  Every Gallai coloring of a complete graph $K_n$ contains a monochromatic spanning tree.
\end{theorem}

\begin{theorem}[Theorem 3.1 of \cites{GySi2004edge}]\label{thm:max} Every Gallai coloring of $K_n$ has a color with largest degree at least $2n/5$.
\end{theorem}

In this paper all logarithms are on base $2$. We omit the floor and ceiling functions as long as they do not affect the calculations.
For any natural number $n$, we denote $\{1, \ldots, n\}$ by~$[n]$. Moreover, whenever we talk about three colors we refer to them as $\red$, $\green$ and $\blue$.

\section{Proof of the main result}
\label{sec:counting-extensions}

Let $\Phi_{n\to k}$ be the set of all Gallai colorings of $K_n$ that use colors in $[k]$. So, each element of $\Phi_{n \to k}$ is a function $\phi : E(K_n) \to [k]$ and $c(n, k) = |\Phi_{n\to k}|$. For any Gallai coloring $\phi$ of $E(K_n)$, we denote by $w(\phi, k)$ the number of ways to extend $\phi$ to a Gallai coloring of $E(K_{n+1})$ where the new edges receive colors from $[k]$ (regardless of how many colors actually appear in $\phi$). 

We start with the following trivial fact that calculates the number of extensions of a monochromatic coloring. This fact and Lemma~\ref{lemma:2t+1} are generalizations of lemmas from \cites{bastos2018counting}.

\begin{fact}\label{claim:monochromatic_extensions}
  Let $k$ and $n$ be a positive integers and $\phi \in \Phi_{n\to k}$ be a monochromatic coloring of the edges of $K_n$.
  Then,
  \[
    w(\phi, k) = (k-1)2^{n} - (k-2).  
  \]
\end{fact}
\begin{proof}
Without loss of generality, assume that all edges of $K_n$ are colored blue. Let $\{u\} = V(K_{n+1}) \setminus V(K_n)$. 
Notice that in any extension of this coloring we can only use one color different from blue.
So, for each choice of the other color (among the other $k-1$ available), we have $2^n$ extensions. 
As the extension in which all edges are blue is counted $k-1$ times, the total number of extensions is $(k-1)2^n - (k-1) + 1$, as claimed.
\end{proof}

\begin{lemma}\label{lemma:2t+1}
  Let $k$ and $n$ be a positive integers and $\phi\in \Phi_{n\to k}$ be any Gallai $k$-coloring of $K_n$. If $\phi'\in \Phi_{n+1\to k}$ is an extension of $\phi$ to $E(K_{n+1})$, then
  \[
   w(\phi',k) \le 2w(\phi,k)+(k-2). 
  \]
\end{lemma}

\begin{proof}
Let $V=V(K_n)$.
  Let $\phi'$ be a extension of $\phi$ to $E(K_{n+1})$ with at most $k$ colors and $u \notin V$ be the new vertex (added to obtain $K_{n+1}$).
  To count the number of Gallai extensions of $\phi'$ to $E(K_{n+2})$, we will add a new vertex $x$ and all edges from $x$ to $V\cup\{u\}$. We first color the edges from $x$ to~$V$.
  If we let $t = w(\phi, k)$, there are $t$ colorings, say $\phi_1, \ldots, \phi_t$, of the edges from $x$ to~$V$. For each $i\in [t]$, we let $m_i$ be the number of ways we can color the edge $ux$ given that we have colored the edges from $x$ to $V$ as in $\phi_i$.
  Clearly, $m_i \in \{0, 1,\ldots, k\}$ and $w(\phi',k) = \sum_{i=1}^{t}m_i$.

  Fix any $i \in [t]$. Recall that the edges from $u$ to $V$ are already colored (in $\phi'$). If there is any vertex $v \in V$ such that $\phi'(xv) \neq \phi_i(uv)$, then the only colors that can be used for $ux$ are $\phi'(xv)$ and $\phi_i(uv)$, then $m_i \le 2$. Therefore, the only way to have $m_i \ge 3$ is when the coloring $\phi_i$ is such that $\phi_i(xy) = \phi(uy)$ for every $y$ in $V$, and for such coloring we have $m_i = k$. This implies that $\sum_{i=1}^{t}m_i\le 2(t-1)+k = 2t + (k-2)$.
\end{proof}

We remark that when $\phi$ is a monochromatic coloring and $\phi'$ is its monochromatic extension, by Fact~\ref{claim:monochromatic_extensions}, we have $w(\phi, k) = (k-1)2^{n}-(k-2)$ and $w(\phi', k) = (k-1)2^{n+1}-(k-2)$. Therefore, $w(\phi',k) = 2w(\phi,k) + (k-2)$, which implies that Lemma~\ref{lemma:2t+1} is best possible.

\subsection{Extensions of graphs that use exactly one color}

We start by proving an easy lemma that says that all Gallai $k$-coloring of $K_n$ have at most as many extensions to a Gallai $k$-coloring of $K_{n+1}$ as the monochromatic colorings. This also implies our weak upper bound on $c(n,k)$. As we shall see later, most of the Gallai $k$-colorings have significantly less extensions than the monochromatic ones, and this fact will imply a better bound for $c(n,k)$. However, the weak upper bound will also be useful in our proof.

\begin{lemma}\label{lemma:uniquemonochromatic_extensions}
  For positive integers $k$ and $n$, and every Gallai coloring $\phi \in \Phi_{n\to k}$, we have 
  \[
    w(\phi, k) \le (k-1)2^{n}-(k-2).
  \]
\end{lemma}
\begin{proof} The proof is by induction on $n$. The result clearly holds for $n=1$ and $n=2$.
   Now, assume that $n\ge 2$, and $w(\phi, k) \le (k-1)2^{n}-(k-2)$ for every $\phi \in \Phi_{n\to k}$. Let $\phi'$ be any Gallai $k$-coloring of $K_{n+1}$. Take any vertex $v \in V(K_{n+1})$ and let $\phi$ be the restriction of the coloring $\phi'$ to the edges of $K_{n+1} - v$. Lemma~\ref{lemma:2t+1} implies that
  \[
    w(\phi', k) \le 2w(\phi, k)+(k-2)     
	\le 2((k-1)2^{n-1}-(k-2))+(k-2) = (k-1)2^{n}-(k-2).\qedhere
  \]
\end{proof}

\begin{remark*}
It also follows from an analogous induction argument that the only colorings that achieve this maximum are the monochromatic ones, but we will not need this fact.
\end{remark*}

\begin{corollary}\label{cor:cnk} For integers $n, k\ge 2$ we have
  $c(n, k) \le (k-1)^n\,2^{\binom{n}{2}}$. 
\end{corollary}
\begin{proof}
  For each fixed $k\ge 2$, we use induction on $n\ge2$. For the base case, when $n=2$, we are simply saying that $k \le (k-1)^2\cdot 2^1$, what is true for $k \ge 2$. Now, assuming the result holds for some $n\ge 2$, by Lemma~\ref{lemma:uniquemonochromatic_extensions}, we have
  \[
    c(n+1, k) = \sum_{\phi \in \Phi_{n\to k}}w(\phi, k) \leq  (k-1)2^{n} \cdot c(n, k) \leq (k-1)^{n+1}\,2^{\binom{n+1}{2}}.\qedhere
  \]

\end{proof}

\subsection{Colorings that are rare of have few extensions}

Due to the results of the previous section, one may hope that the $k$-colorings that are ``close to being monochromatic'' are those that have the most extensions. Our proof of Theorem~\ref{thm:main_result} does \emph{not} require us to prove precisely this. But we do show that whenever we have a \emph{$2$-coloring} where both colors are used many times, then we have few extensions. 

\begin{lemma}\label{lemma:nonextremal} 
For every $m, n, k\ge 2$, every 2-colorings $\phi$ of $K_n$ such that both colors are used more than $3mn$ times satisfies
   \[
   w(\phi, k) \leq 2^n + kn 2^{n  - 0.4 m}.
  \]
  Furthermore, at most $kn 2^{n  - 0.4 m}$ such extensions use a color that is not used in $\phi$.
\end{lemma}
\begin{proof}
  For $k=2$ the result is trivial, since there are at most $2^n$ extensions. 
Let $k \ge 3$.
  Let $\phi $ be a coloring of $E(K_n)$ by $\red$ and $\blue$ as in the statement.
  As before, let $u \in V(K_{n+1})\setminus V(K_n)$, so that we want to count the number of ways to color the edges from $u$ to $V(K_n)$. 
 First, note that there are $2^n$ ways to extend $\phi$ to a Gallai coloring of $K_{n+1}$ using only red and blue.
Secondly, note that using two new colors (that is, different from red and blue) will immediately create a rainbow triangle.
Thus, it remains to show that there are at most $kn 2^{n  - 0.4 m}$ extensions that use exactly one new color, say green.
  
We claim that there exist $\{(r_1, u_1, b_1), (r_2, u_2, b_2), \ldots, (r_{m+1}, u_{m+1}, b_{m+1})\}$, a set of disjoint triples of vertices of $K_n$ such that $\phi(u_ir_i) = \red$ and $\phi(u_ib_i) = \blue$ for all $i\in [m+1]$. 
Indeed, assume that the maximum set of such disjoint triples has size $m$.
This implies that all edges inside $V(K_n) - \bigcup_{i = 1}^{m} \{r_i, u_i, b_i\}$ have the same color, say blue. So, the number of red edges is at most $3mn$, a contradiction.
  
  Fix a vertex $v \in V(K_n)$. Suppose that we want to build an extension $\phi'$ of $\phi$ such that $\phi'(uv) = \green$. Let us count in how many ways we can complete the extension $\phi'$. Note that, for any vertex $w\in V(K_n)\setminus \{v\}$, (as $\phi(vw)$ is not green) $\phi'(uw)$ must be either green or $\phi(vw)$. So, assuming $\phi'(uv) = \green$, there are at most $2$ choices for every other edge.  However, we claim that for every $i \in [m+1]$ such that $v\notin \{u_i, r_i, b_i\}$, there are at most $6$ ways to color the set of edges $\{uu_i, ur_i, ub_i\}$. 

To see this, fix $i$ such that $v\notin \{u_i, r_i, b_i\}$ and assume that $\phi(vu_i) = \red$ (the case $\phi(vu_i) = \blue$ is analogous). 
Recall that $ub_i$ must receive either $\green$ or $\phi(vb_i)$. 
In the case $ub_i$ is $\green$, then $uu_i$ must also be $\green$ (looking at the triangles $uvu_i$ and $ub_iu_i$) and we have (at most) $2$ options for the color of $ur_i$. 
In the case $ub_i$ has color $\phi(vb_i)$, we trivially have at most $4$ options for the colors of $uu_i$ and $ur_i$. This gives a total of at most $6$ ways to color the set $\{uu_i, ur_i, ub_i\}$.

  Finally, we count how many extensions $\phi$ has with one new color (with some room to spare). We have $n$ options to choose a vertex $v$, have $k-2$ options for the color of $uv$, have $6^{m}$ ways to color the edges from $u$ to $m$ of those $\{u_i, r_i, b_i\}$ such that $v\notin \{u_i, r_i, b_i\}$ and $2$ ways to color each of the $n-3m-1$ remaining edges. This gives less than
  \begin{equation*}
        kn6^{m} 2^{n - 3m-1}  =  kn 2^{n  - (3 - \log 6)m-1} \le kn 2^{n  - 0.4 m}
  \end{equation*}
  such extensions.
\end{proof}

Let $F(n,k)$ be the set of colorings of $K_n$ with colors from $[k]$, such that one color forms a spanning subgraph of $K_n$ and the others span pairwise vertex-disjoint (possibly empty) subgraphs of $K_n$. By Theorem~\ref{thm:GyarfasTree}, a Gallai $k$-coloring is in $F(n,k)$ if, and only if, it does not have a vertex that has three edges of different colors incident to it. Let $F'(n,k)\subseteq F(n,k)$ be the set of such colorings with the extra assumption that there is no set of at  least~$0.9n$ vertices that induces a 2-colored clique. Let $f'(n,k) = |F'(n,k)|$.
Next, we give an upper bound on $f'(n,k)$.

\begin{lemma}\label{lem:9} For every $n, k \ge 2$, we have 
\[
  f'(n,k) \le 2^{\binom n2 - 0.05n^2 + (n+1)\log k}.
\]
\end{lemma}

\begin{proof}
There are $k$ choices for the color of the spanning tree. Assume, without loss of generality that we have chosen color $k$ for it. There are $(k-1)^n < k^n$ ways to partition the vertices of $K_n$ into (labeled and possibly empty) classes. For $i\in [k-1]$, let $x_i$ be the number of vertices in class~$i$. We have that $\sum_{i\in[k-1]} x_i= n$ and there are $2^{\sum_{i\in [k-1]} \binom{x_i}{2}}$ ways\footnote{For the sake of this notation, we are considering $\binom 02 = \binom12 = 0$.}  to color the edges of $K_n$ so that those inside class $i$ receive color $i$ or color $k$, for every $i\in[k-1]$, and edges between classes receive color $k$. Note that
\begin{equation*}
\sum_{i\in [k-1]} \binom{x_i}{2} = \binom n2 - \sum_{1\le i < j\le k-1} x_i x_j = \binom n2 - \frac12 \sum_{i\in [k-1]} x_i (n-x_i).
\end{equation*}
By the definition of $F'(n,k)$ we have $x_i \le 0.9n$ for every $i\in[k-1]$. Therefore,
\[
  \sum_{i\in [k-1]} x_i (n-x_i) \ge \sum_{i\in [k-1]} x_i (0.1n) = 0.1n^2.
\]
Thus, 
\[
\sum_{i\in [k-1]} \binom{x_i}{2} \le \binom n2 - 0.05n^2.
\]
In total, we get
\[
f'(n,k) \le k\cdot k^n 2^{\binom n2 - 0.05n^2} = 2^{\binom n2 - 0.05n^2 + (n+1)\log k}.  \qedhere
\]
\end{proof}

\medskip
The next lemma treats another case in which we can guarantee that there are few extensions: Gallai $k$-coloring of $K_n$ that contains no large subgraph that induces a coloring in~$F$. For $k$ in the range of Theorem~\ref{thm:main_result}, the number of such extensions is significantly less than the one for monochromatic colorings.

  \begin{lemma}\label{lemma:nobig2colored}
    Let $k, t, n\ge 2$ be integers such that $n\ge 75 t$, and $\phi \in \Phi_{n \to k}$ be a Gallai $k$-coloring that does not contain any set of ${n}/{3}$ vertices which induce a coloring in $F(n/3, k)$. Then
    \[w(\phi, k) \leq t k^2 2^{n-t}.\]
  \end{lemma}

  \begin{proof}
    Let $\phi$ be a Gallai $k$-coloring of $K_n$ as in the statement of this lemma.
    We start proving that $\phi$ must contain a certain structure. 
    By Theorem~\ref{thm:max}, we can inductively choose vertices $v_1,\dots, v_{t}$, such that for every $i\in [t]$, there is a color $c_i$ such that $v_i$ is adjacent to at least $2(n-i+1)/5 \ge 2n/5 - t$ vertices in $K_n\setminus\{v_1,\dots, v_{i-1}\}$ through edges of color $c_i$. Let $T=\{v_1,\dots, v_{t}\}$.


    Next, for each $i\in [t]$, we will build $t$ vertex-disjoint rainbow copies of $K_{1,3}$ (where different copies may use different triples of colors)  in $N_{c_i}(v_i)\setminus T$, where $N_{c_i}(v_i)$ stands for the set of vertices that are connected to $v_i$ via edges of color $c_i$.
    Moreover, we do not care whether the copies of $K_{1,3}$ for different $i$'s are disjoint or not.

    Indeed, we shall find them greedily. Suppose we have chosen some of those copies of $K_{1,3}$'s and we want to find an extra one in $N_{c_i}(v_i)\setminus T$, for a desired $i \in [t]$. In total, the current copies take at most $4t$ vertices. Thus, there are at least $2n/5 - |T| - 4t = 2n/5-5t \ge n/3$ vertices available in $N_{c_i}(v_i)\setminus T$. 
    Let $A$ be the set of those vertices. If any of them is incident to 3 edges of distinct colors with all endpoints also in $A$, then we are done. 
    Moreover, by Theorem~\ref{thm:GyarfasTree}, applied to the coloring induced by $A$, at least one color class spans $A$. 
    Thus, if we cannot find the desired extra copy of rainbow $K_{1,3}$, then the remaining vertices of $A$ induce a coloring in $F(m,k)$ with $m\ge n/3$, a contradiction.

    Now, we prove the upper bound for $w(\phi, k)$. Consider a new vertex $u$ and let us count in how many ways we may color the edges from $u$ to $K_n$. 

    First assume that for some $i\in [t]$, the edge $u v_i$ receives a color different from $c_i$, say $c_{u, i}$. Then consider the edges in the rainbow $K_{1,3}$'s contained in $N_{c_i}(v_i)$ that have a color different from $c_i$ and $c_{u, i}$. Each $K_{1,3}$ has at least one such edge, therefore, we can select a matching $M$ of size $t$ formed by those edges. Let us denote it by $M = \{a_1 b_1, \dots, a_t b_t\}$. 
    By considering the quadruple $u, v_i, a_j, b_j$, for $j\in [t]$, it is easy to see that the colors of $u a_j$ and $u b_j$ must be equal and be either $c_i$ or $c_{u, i}$. Thus, there are at most $2^t$ ways to color the $2t$ edges $u a_j$ and $u b_j$, where $j\in [t]$. By Lemma~\ref{lemma:uniquemonochromatic_extensions}, we can color the remaining edges from $u$ to $K_n$ in less than $(k-1)2^{n-2t-1}$ ways. Summing over all $i\in [t]$ and the choice of the color of $uv_i$, we obtain at most 
    \[
      t \cdot (k-1)\cdot 2^t \cdot (k-1)2^{n-2t-1} = t(k-1)^2\, 2^{n-t-1}
    \] extensions.

    It remains to consider the case when for all $i\in [t]$, the edge $u v_i$ receives color $c_i$. In this case, by Lemma~\ref{lemma:uniquemonochromatic_extensions}, we can color the other edges from $u$ to $K_n$ in less than $(k-1)2^{n-t}$ ways. 

    In total, we have 
$w(\phi, k) \le t(k-1)^2 2^{n-t-1} + (k-1)2^{n-t} \le {t k^2 2^{n-t}}$.
    \end{proof}

\medskip
Now we are ready to prove Theorem~\ref{thm:main_result}. The idea of the proof is that any Gallai coloring of $K_n$ can be treated as an extension of a coloring in $F(a,k)$ for some maximum $a$.
Furthermore, we need to consider the largest 2-colored clique of the coloring on such $a$ vertices. 

\subsection{Proof of Theorem~\ref{thm:main_result}}\label{sec:proof-maintheorem}

For each coloring $\phi \in \Phi_{n\to k}$, let $A(\phi)$ be a set of vertices of \emph{maximum} size such that the restriction of $\phi$ to $K_n[A(\phi)]$ forms a coloring in $F(|A(\phi)|, k)$ (in case there is more than one choice for $A(\phi)$, select one arbitrarily). 
Furthermore, let $A'(\phi)\subseteq A(\phi)$ be a set of vertices of \emph{maximum} size such that $K_n[A'(\phi)]$ is 2-colored. Let $a: = a(\phi) = |A(\phi)|$ and $m: = m(\phi)= |A'(\phi)|$. 

We say that a $2$-coloring of $E(K_m)$ is nearly \emph{nearly monochromatic} if one of the colors is used at most $m^2/20$ times. Consider the following sets of colorings.
  \begin{align*}
    M_1 &= \left\{\phi \in \Phi_{n\to k} \colon a(\phi) <  \frac{n}{6}\right\}, \\
    M_2 &= \left\{\phi \in \Phi_{n\to k} \colon n> m(\phi) \geq \frac{n}{7} \text{ and $A'(\phi)$ is not nearly monochromatic}\right\}, \\
    M_3 &= \left\{\phi \in \Phi_{n\to k} \colon n> m(\phi) \geq \frac{n}{7} \text{ and $A'(\phi)$ is nearly monochromatic}\right\} \text{ and}  \\
        M_4 &= \left\{\phi \in \Phi_{n\to k} \colon a(\phi) \geq \frac{n}{6} \text{ and } m\le n/7\right\}.
  \end{align*}

Note that $\Phi_{n\to k}\setminus (M_1 \cup M_2 \cup M_3\cup M_4)$ is the set of the $k$-colorings of $E(K_n)$ that use at most two of the colors (that is, $m=n$).
Thus, we have $|\Phi_{n\to k}\setminus (M_1 \cup M_2\cup M_3\cup M_4)|  < \binom{k}{2}2^{\binom{n}{2}}$.
To prove Theorem~\ref{thm:main_result}, it remains to show that $|M_i|\le 2^{\binom n2} o_n(1)$ for each $i\in [4]$. The main idea to bound $|M_1|$ (resp. $|M_2|$) is that although there are many options for how we choose and color $A(\phi)$ (resp. $A'(\phi)$), those coloring can be extended to a coloring of $K_n$ in few ways. On the other hand, to bound $|M_3|$ (resp. $|M_4|$) there are so few ways to color set $A'(\phi)$ (resp. $A(\phi)$) that even if we use the general bound on Lemma~\ref{lemma:uniquemonochromatic_extensions} to count the extensions of theses colorings, we get few colorings in $M_3$ (resp. $M_4$).
%

\textbf{Upper bound for $|M_1|$.} Fix an arbitrary ordering, say $(v_1, \ldots, v_n)$, of the vertices of $K_n$ and, for each $j \in [n]$, let $K_j$ be the graph induced by $\{v_1, \ldots, v_j\}$.

Let $s = \lceil n/2 \rceil$. We start with a Gallai $k$-coloring of $K_s$ and count in how many ways it can be extended to a coloring of $E(K_n)$ that belongs to $M_1$. 
By Corollary~\ref{cor:cnk}, there are less than $k^{s}2^{\binom{s}{2}}$ such colorings of $E(K_s)$.
Now, for each $j$ from $s$ to $n-1$, since we want to count colorings in $M_1$, keep only those colorings of $E(K_j)$ that do not have a set of $n/6$ vertices which induces a coloring in $F(n/6, k)$. As $n/6 \le j/3$, they also do not have $j/3$ vertices that induce a coloring in $F(j/3, k)$. By Lemma~\ref{lemma:nobig2colored} with $j$ (in place of $n$ in Lemma~\ref{lemma:nobig2colored}) and $t=n/150\le j/75$, we conclude that there are at most $tk^2 2^{j-t}$ extensions from $K_j$ to $K_{j+1}$. Using that $2^{\binom{s}{2}}\prod_{j=s}^{n-1}2^j = 2^{\binom{n}{2}}$, $k\le 2^{n/500}$ and $n$ is large enough, we have
  \begin{align*}
    |M_1| &\leq k^s2^{\binom{s}{2}}
                \bigg(\prod_{j = s}^{n-1}tk^2 2^{j-t}\bigg) \leq 2^{\binom{n}{2}} k\left(\frac{t k^3}{2^t}\right)^{\!\lfloor n/2 \rfloor} = 2^{\binom{n}{2}} k\left(\frac{n k^3}{150\cdot 2^{n/150}}\right)^{\!\lfloor n/2 \rfloor}
            = 2^{\binom{n}{2}} o_n(1).
  \end{align*}

\textbf{Upper bound for $|M_2|$.} For each $m$ with $n/7 \leq m < n$, there are $\binom{n}{m}\le n^{n-m}$ ways to choose a set $S$ of $m$ vertices as a candidate for $A'(\phi)$ and less than $\binom{k}{2}2^{\binom{m}{2}}$ ways to color the edges induced by it. Consider only those $2$-colorings of the edges in $S$ that are not nearly monochromatic, that is, both colors are used more than $m^2/20$ times.

This time the counting of extensions of such colorings to $K_n$ is done in a different way. Let $V(K_n)\setminus S = \{v_{m+1}, \ldots, v_n\}$. We have to color all edges incident to those vertices. 

For $i$ varying from $m+1$ to $n$, first we color the edges from $v_i$ to $S$ and then we color the edges from $v_i$ to $\{v_{m+1}, \ldots, v_{i-1}\}$. 
Notice that, by the maximality of $m$, the edges from $v_i$ to $S$ must use one color different from the colors on $S$. 
Therefore, by Lemma~\ref{lemma:nonextremal}, there are at most $km{2^{m-0.4m/60}} = km{2^{m- m/150}}$ ways to color those edges.
Moreover, by Lemma~\ref{lemma:uniquemonochromatic_extensions} there are at most 
$k2^{i-m-1}$ ways to color the edges from $v_i$ to $\{v_{m+1}, \ldots, v_{i-1}\}$.
Thus, as $k\le 2^{n/4300}$ and $n$ is large enough, we have
  \begin{align*}
    |M_2| &\leq \sum_{m={n}/{7}}^{n-1} n^{n-m} \binom{k}{2}2^{\binom{m}{2}}\left(\prod_{i = m}^{n-1}  km{2^{m- m/150}} k2^{i - m}\right) \\
           &\leq {k}^{2} 2^{\binom{n}{2}} \sum_{m={n}/{7}}^{n-1} n^{n-m} \left(\frac{k^2 m}{2^{m/150}}\right)^{n - m}                      \\
           &\leq {k}^{2} 2^{\binom{n}{2}} \sum_{m={n}/{7}}^{n-1} \left(\frac{n^2 k^2}{2^{n/1050}}\right)^{n-m}      \leq 2^{\binom{n}{2}} \frac{n^3 k^4}{2^{n/1050}}  = 2^{\binom{n}{2}} o_n(1).
  \end{align*}

\textbf{Upper bound for  $|M_3|$.} For $m$ with $n/7 \leq m < n$, we have $\binom{n}{m} < 2^n$ ways to choose a set $S$ of $m$ vertices.
We give an upper bound for how many colorings $\phi \in M_3$ are such that $A'(\phi) = S$.
First, there are $k(k-1)\le k^2$ ways to choose the $2$ colors for the edges in $A'(\phi)$ and to select one color to be the less frequent one of them.
 Let $c$ be the number of edges with the less frequent color in $A'(\phi)$. 
 Since $A'(\phi)$ is nearly monochromatic, we have $c \leq m^2/20$. 
For $0\le c\le m^2/20$, we have
  \[
 \binom{\binom{m}{2}}{c}\le \binom{m^2/2}{c} \le \binom{m^2/2}{m^2/20} \le (10e)^{m^2/20}.
  \]
Thus, as $m\ge n/7$ is large enough, the number of ways to color the edges induced by $S$ is at most
\[
k^2 \sum_{c=0}^{m^2/20}\binom{\binom{m}{2}}{c} \le k^2 \left(\frac{m^2}{20}+1\right) (10e)^{m^2/20} \le k^2 2^{m^2/4},
\]
because $\log (10 e)\le 4.8$.
This is already small enough that, to bound $|M_3|$, we simply use Lemma~\ref{lemma:uniquemonochromatic_extensions} for each $j$ from $m$ to $n-1$ (to count in how many ways we can extend each coloring of~$S$ to $K_n$). 
Thus, as $k\le 2^{n/200}$ and $n$ is large enough,
we have
  \begin{align*}
    |M_3| &\leq \sum_{m={n}/{7}}^{n-1} 2^n k^2 2^{m^2/4}\prod_{j = m}^{n-1}k2^j                               \\
           &\leq \sum_{m={n}/{7}}^{n-1} 2^{n + m^2/4 }k^{n}\, 2^{\binom{n}{2} - \binom{m}{2}}                     \\
           &\leq 2^{\binom{n}{2}}\sum_{m={n}/{7}}^{n-1} 2^{n + n\log k + m^2/4 - \binom{m}{2}}                    \\
           &\leq 2^{\binom{n}{2}} n 2^{n + n^2/200 + (n/7)^2/4 - \binom{{n}/{7}}{2}} = 2^{\binom{n}{2}}o_n(1).
  \end{align*}

\textbf{Upper bound for  $|M_4|$.}
Our counting this time is similar as that for $|M_3|$.
For $a$ such that $n/6\le a \le n$, we have $\binom{n}{a} < 2^n$ ways to choose a set $S$ of $a$ vertices.
Now, we bound the number of colorings $\phi \in M_4$ such that $A(\phi) = S$.
By the definition of $M_4$ (as $6/7 < 0.9$) and Lemma~\ref{lem:9}, we know that for every $\phi\in M_4$, there are $f'(a,k)\le 2^{\binom a2 - 0.05a^2 + (a+1)\log k}$ possibilities for the coloring of $A(\phi)$.
We then use Lemma~\ref{lemma:uniquemonochromatic_extensions} for each $j$ from $a$ to $n-1$ to count in how many ways we can extend each coloring of~$S$ to $K_n$. 
Thus, 
we obtain
  \begin{align*}
    |M_4| &\leq \sum_{a={n}/{6}}^{n-1} 2^n 2^{\binom a2 - 0.05a^2 + (a+1)\log k}\prod_{j = a}^{n-1}k2^j                               \\
           &\leq 2^{\binom{n}{2}} \sum_{a={n}/{6}}^{n-1} 2^{n - 0.05a^2 + (a+1)\log k}k^{n}\,  \\
           &\leq 2^{\binom{n}{2}}\sum_{a={n}/{6}}^{n-1} 2^{n - 0.05a^2 + (a+n+1)\log k} .  
  \end{align*}
As $k \le 2^{n/60}$, the function $h(a) = - 0.05a^2 + a\log k + (n+1)\log k$ is decreasing on $n/6 \le a$.
And, as $k\le 2^{n/900}$ and $n$ is large enough, we have
\[
n - 0.05a^2 + (a+n+1)\log k \le n - \frac{0.05}{36} n^2 + \left(\frac 76 n+1\right)\frac n{900}\le -n.
\]
Therefore, we have $|M_4|\le 2^{\binom{n}{2}} n 2^{-n} = 2^{\binom{n}{2}} o_n(1)$.

The proof of Theorem~\ref{thm:main_result} is completed.
  
 \medskip
\noindent\textbf{Acknowledgement.}
We thank Carlos Hoppen, Guilherme O. Mota and Maur\'icio Neto for helpful discussions.


\bibliographystyle{amsplain}

\end{document}